\documentclass{amsart}
\usepackage{amsmath}
\usepackage{amsthm}
\usepackage{amssymb}
\usepackage[curve]{xypic}
\usepackage{dsfont}
\usepackage{latexsym}
\usepackage{varioref}
\originalTeX
\usepackage{cite}
\usepackage[T1]{fontenc}
\usepackage[utf8]{inputenc}
\usepackage{overpic}

\def\IZ{\mathds{Z}}
\def\IN{\mathds{N}}
\def\IR{\mathds{R}}
\def\IC{\mathds{C}}

\def\IF{\mathds{F}}

\def\Inv{\mathrm{Inv}}

\def\eps{\varepsilon}
\def\iso{\simeq}

\def\e{\mathrm{e}}

\newcommand\abs[1]{\mathord{\left\lvert#1\right\rvert}}
\newcommand\norm[1]{\mathord{\left\Vert#1\right\Vert}}

\newcommand{\undefined}{\Diamond}

\DeclareMathOperator{\interior}{int}

\DeclareMathOperator{\cl}{cl}

\DeclareMathOperator{\Def}{\mathcal{D}}

\DeclareMathOperator{\Con}{\mathcal{C}}

\DeclareMathOperator{\Hom}{\mathrm{H}}
\DeclareMathOperator{\dirlim}{\mathrm{dirlim}}

\DeclareMathOperator{\Diff}{D}
\newcommand{\grad}{\nabla}

\newcommand{\Hull}{\mathcal{H}}

\newcommand{\rep}[1]{\left[#1\right]}

\newtheorem{theorem}{Theorem}[section]
\newtheorem{lemma}[theorem]{Lemma}

\newtheorem{corollary}[theorem]{Corollary}

\theoremstyle{definition}
\newtheorem{definition}[theorem]{Definition}

\theoremstyle{remark}

\newcommand{\imag}{\mathrm{i}}

\begin{document}
	\title[Equilibrium-like solutions of asymptotically autonomous equations]{Equilibrium-Like Solutions of Asymptotically Autonomous Differential Equations}
	
	\author[A. J\"anig]{Axel J\"anig}
	
	\address
	{\textsc{Axel J\"anig}\\
		Institut f\"ur Mathematik\\
		Universit\"at Rostock\\
		18051 Rostock, Germany}
	
	\email{axel.jaenig@uni-rostock.de}
	
	\subjclass[2010] {Primary: 37B30, 37B55; Secondary: 34C99, 35B40, 35B41}
	\keywords{nonautonomous differential equations, asymptotically autonomous differential equations, 
		Morse decompositions, chain-recurrent sets,
		Morse-Conley index theory, nonautonomous Conley index, homology Conley index}
	
	\begin{abstract}%
		We analyze the chain recurrent set of skew product semiflows obtained from nonautonomous
		differential equations -- ordinary differential equations or semilinear parabolic differential
		equations. For many gradient-like dynamical systems, Morse-Smale dynamical systems e.g., the
		chain recurrent set contains only isolated equilibria. 	
		The structure in the asymptotically autonomous setting is richer but still close 
		to the structure of a Morse-Smale dynamical system. 
		
		The main tool used in this paper is a nonautonomous flavour of Conley index theory
		developed by the author. We will see that for a class of good equations, the Conley index
		can be understood in terms of equilibria (in a generalized meaning) and their connections.
		This allows us to find specific solutions of asymptotically autonomous equations and generalizes
		properties of Morse-Smale dynamical systems to the asymptotically autonomous setting.
	\end{abstract}
	
	\maketitle
	
	A gradient-like dynamical system has only two types of solutions: equilibria
	or heteroclinic connections. Prototypical examples are given by differential equations of the form
	\begin{equation}
		\label{eq:171121-1127}
		\dot x = \grad \Phi(x)
	\end{equation}
	
	Assume that each of these equilibria is hyperbolic, that is, for every equilibrium $x_0$
	it holds that all eigenvalues of the derivative $\Diff(\grad \Phi(x_0))$ have nonzero real parts.
	\eqref{eq:171121-1127} could be perturbed by a nonautonomous function $f(t,x)$ converging to
	$0$ uniformly on compact sets as $t\to\infty$.
	
	\begin{equation}
		\label{eq:171121-1159}
		\dot x = \grad \Phi(x) + f(t,x)
	\end{equation}
	
	A hyperbolic equilibrium admits an exponential dichotomy. Indeed, hyperbolicity can be characterized
	by the existence of exponential dichotomies. Hence, if the perturbation $f$ is small enough,
	is exactly one entire or full solution\footnote{A solution $u$ is called full or entire if $u$ is defined for all $t\in\IR$.} solution of \eqref{eq:171121-1159} in a small neighbourhood of the original equilibrium $x_0$. This solution $u$ converges
	to $x_0$ as $\abs{t}\to\infty$, so it seems that the role of an equilibrium is played by connections
	between equilibria
	
	We will not define the notion of a nonautonomous equilibrium. It is not uncommon to define
	nonautonomous equilibria as an entire bounded solution. This is compatible with the findings
	in this paper as long as one is interested in hyperbolic solutions.
	
	If $f$ is more than a small perturbation of a nonautonomous problem, one can no longer
	expect a one-to-one correspondence between equilibria of the nonautonomous problem and
	its autonomous limit. Recall that, for an autonomous problem,
	the number of equilibria might be related to the Betti numbers of a manifold or the
	asymptotic behaviour of $\grad \Phi$. Many of these results can be proved by using
	the so-called Conley index. This might not be the only technique, and the index might not always
	be applicable but the approach is often a successful. Since there is a nonautonomous extension of the Conley index,
	we would like to study the index in relation to equilibrium solutions, by which we mean
	heteroclinic solutions connecting two equilibria of the same Morse index.
	
	The Conley index cannot be applied directly to \eqref{eq:171121-1159}. Furthermore, it would also
	be interesting to include equations which are asymptotically autonomous but with different
	limits at $-\infty$ and $\infty$. Assume that $E$ is a finite-dimensional normed space,
	and let $Y$ denote the space of all $f:\;\IR\times E\to E$ being continuous in 
	both variables, continuously differentiable and Lipschitz-continuous in the second one. 
	$Y$ is equipped with a metric such that
	$f_n\to f$ in $Y$ if and only if $f_n\to f$ uniformly on compact subsets of $\IR\times E$.
	A function $f\in Y$ is said to be autonomous if $f(t,x) = f(0,x)$ for all $(t,x)\in\IR\times E$.
	$f$ is asymptotically autonomous if there are autonomous $f^{\pm\infty}\in Y$ such that
	$d(f^t, f^{\pm\infty})\to 0$ as $t\to\pm\infty$, where $f^s(t,x) := f(t+s,x)$ denotes the translation
	in the time variable, which defines a flow on $Y$.
	
	To apply the nonautonomous Conley index, an initial element is required, that is, an $f_0\in Y$
	such that $\omega(f_0) = \Hull(f)$, where $\omega(f_0)$ denotes the usual $\omega$-limes set with
	respect to the translational flow on on $f$ and $\Hull(f) := \cl_Y\{f^t:\;t \in\IR\}$ is the hull of $f$.
	It is easy to see that an initial element $f_0$ can only exists if $f^{-\infty} = f^{\infty}$.
	The situation is sketched in the diagram below.
	
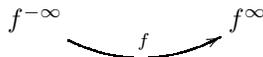
\begin{equation*}
		\xymatrix@C=2cm{
			f^{-\infty} \ar@/_5mm/[r]^-{f} & f^\infty
		}
	\end{equation*}
	The missing link, so to speak, is a reverse connection from $f^\infty \to f^{-\infty}$.
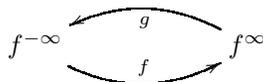
	\begin{equation*}
	\xymatrix@C=2cm{
		f^{-\infty} \ar@/_5mm/[r]^-{f} & \ar@/_5mm/[l]^-g f^\infty
	}
	\end{equation*}
	In many cases, $g(t,x) := f(-t,x)$ is a useful choice. Given a circle
	as above, $f_0$ can be chosen as a rotation in the above diagram which becomes
	slower and slower at $f^{-\infty}$ and $f^{\infty}$ as $t\to\infty$, so $\omega(f_0) = \Hull(f)\cup \Hull(g)$.
	
	Suppose $K\subset \omega(f_0) \times E$ is a compact invariant subset. Without
	additional assumptions, we know $K$ is not empty unless its index is trivial, but we cannot say much more about $K$.
	That is why we focus on a class of good problems.
	These assumptions consist of two parts. The first one concerns the
	limit equations determined by $f^\infty$ and $f^-\infty$, imposing a weak version
	of a Morse-Smale condition that is, (1) and (2) are in particular satisfied if the limit equations
	are Morse-Smale, but transverse intersection of stable and unstable manifolds is not required.
	
	\begin{enumerate}
		\item[(1)] There are only finitely many equilibria, and each of them is hyperbolic.
		\item[(2)] If $u:\;\IR\to E$ is a solution of
		\begin{equation*}
			x_t = f^{\pm\infty}(0,x)
		\end{equation*}		
			with $u(t)\to e^\pm$ as $t\to\pm\infty$, then $m(e^+)<m(e^-)$, where $m(e)$
			denotes the Morse index of the equilibrium $e$, that is, the dimension of its
			local unstable manifold.
	\end{enumerate}

	The second part of the assumptions 
	concerns the nonautonomous problem. It is clear that a bounded solution $u:\;\IR\to E$
	of \eqref{eq:171121-1404} 
	converges to equilibria $e^\pm$ as $t\to\pm\infty$ by virtue of (1) and (2). 	
	
	\begin{enumerate}
		\item[(3)] If $u:\;\IR\to E$ is a solution of
		\begin{equation}
		\label{eq:171121-1404}
		x_t = f(t,x)\quad \text{(or $g(t,x)$)}
		\end{equation}		
		with $u(t)\to e^\pm$ as $t\to\pm\infty$, then $m(e^+)<m(e^-)$, where $m(e)$
		denotes the Morse index of the equilibrium $e$.
		
		Moreover, $m(e^-) = m(e^+)$ only if $u$ is weakly hyperbolic that is, 
		\begin{equation*}
			x_t = \Diff_x f(t,u(t))x \quad \text{(or $\Diff_x g(t,u(t))x)$}
		\end{equation*}
		does not have a nontrivial bounded solution.
	\end{enumerate}
	
	On $C^k(\IR\times E, E)$, we consider a topology induced by a family $(\delta_{n,k})$ of seminorms:
	\begin{align*}
			\delta_{n,0}(f) &:= \sup_{(t,x) \in \IR\times B_n(0; E)} \norm{f(t,x)}_{E}\\
			\delta_{n,k}(f) &:= \sup_{(t,x) \in \IR\times B_n(0; E)} \norm{\Diff^k f(t,x)}\quad l\in\{1,\dots,k\}
	\end{align*}
	A neighbourhood sub-basis\footnote{$N\subset C^k(\IR\times E, E)$ is a neighbourhood of $f$ iff there are finitely many
		sub-basis elements $N_1,\dots,N_k$ which are neighborhoods of $f$ and whose intersection is a subset of $N$.} for an element $f\in C^k(\IR\times E, E)$ is formed by sets of the form
	\begin{equation*}
	  N_{n,l,\eps} := \{ g\in C^k(\IR\times E, E):\; \delta_{n,l}(f-g)<\eps\}\quad (n,l,\eps)\in\IN\times\{0,\dots,k\}\times\IR^+
	\end{equation*}
	A generic\footnote{The set of all such $h$ is
		residual i.e., a superset of a countable intersection of open and dense subsets.}
	asymptotically autonomous parameter $f \in C^k(\IR\times E, E)$ satisfying (1) and (2) 
	is good in the sense that (3) holds.
	
	
	We say that $u:\;\IR\to E$ is a solution of Morse-index $m$ if
	$u$ is an equilibrium solution of $f^\infty$ or $f^{-\infty}$
	with Morse-index $m$ or if $u$ is a solution of \eqref{eq:171121-1404}
	converging to equilibria of Morse-index $m$. Setting 
	\begin{equation*}
		M_k := \{(f',u):\; u\text{ is a solution of Morse-index } m\} \cap K
	\end{equation*}
	the family $(M_k)_{k\in\IN}$ is a Morse-decomposition of $K$ with respect
	to the skew product semiflow $\pi$ on $Y\times E$ defined by $(y,x)\pi t := (y^t, \Phi_y(t,0,x)))$,
	where $\Phi_y(t,0,x):=u(t)$ is the solution of the nonautonomous equation $\dot x = y(t,x)$ with $u(0)=x$.
	Hence, for every solution $u:\;\IR\to Y\times E$ of $\pi$, the following alternative holds:
	Either there is a $k_0\in\IN$ such that $u(t)\in M_{k_0}$ for all $t\in\IR$ or there $k<l$ in $\IN$
	such that $\alpha(u)\subset M_l$ and $\omega(u)\subset M_k$, where $\alpha$ and $\omega$ denote
	the usual $\alpha$- and $\omega$-limes sets.
	
	While $(M_k)_k$ is a Morse-decomposition, it is probably not the finest 
	Morse-decomposition\footnote{Conley \cite{conley1978isolated} points out that a finest Morse-decomposition need not exist because the would-be finest Morse-decomposition might
		consist of infinitely many Morse sets. The union of all finest Morse-sets is then called chain recurrent set.
		In our case there are finest Morse-sets, so the distinction is meaningless.}.
	For example if $f$ was only a very small perturbation of an autonomous equation,
	there would be a single Morse-set for each equilibrium of the unperturbed equation.
	Let us say that $M_e$ is the Morse-set belonging to an equilibrium $e$ of Morse-index $m(e)$.
	The homology Conley index is given by the following formula which is well-known in the autonomous case.
	\begin{equation}
		\label{eq:171122-1341}
		\Hom_q\Con(f_0,M_e) \iso \begin{cases} \IZ & q = m\\
		0 & q\neq m \end{cases}
	\end{equation}
	
	This paper is an attempt to answer the following question. Does the above equation still hold for a general asymptotically autonomous $f\in Y$ if we replace $M_e$ by a finest Morse-set?
	
	The answer is a little more complicated than in the autonomous case. 
	First of all in Section \ref{sec:trivial}, we will construct an example
	having a finest Morse set $M$ with $M(f) := \{x:\; (f,x)\in M\}\neq \emptyset$ yet $\Hom_q\Con(f_0, M) = 0$
	for all $q\in\IZ$, so there are, roughly speaking, equilibria with trivial index. Second, using homology with coefficients over a field $\IF$, we find a formula (a straightforward corollary to Theorem \ref{th:171115-1441}) for arbitrary isolated invariant sets $K\subset \omega(f_0)\times E$.
	\begin{equation}
		\label{eq:171206-1429}
		\dim \Hom^\IF_q\Con(f_0, K) \leq \min\{N_q(f^{-\infty},K(f^{-\infty})), N_q(f^\infty, K(f^\infty))\}
	\end{equation}
	Here, $N_q(f^{\pm\infty},K(f^{\pm\infty}))$ denotes the number of equilibria $e$
	of $\dot x = f^{\pm\infty}(t,x)$ with $e\in K(f^{\pm\infty})$ and Morse-index $m(e)=q$.
	
	Since the formula holds for arbitrary compact invariant sets, it is still also for a finest Morse-set $M$.
	In particular, if $M\subset M_k$, then
	\begin{equation*}
		\Hom^\IF_q\Con(f_0, M) = 0 \quad \text{for } q\neq k 
	\end{equation*}
	
	Summing up, in this nonautonomous but asymptotically autonomous setting, there are Morse-sets
	behaving similar to an equilibrium of an autonomous equation. Although there are also phenomena not encountered
	in gradient-like dynamics, every compact invariant set can be decomposed into finest Morse-sets
	of a single dimension and their connections. 
	
	The structure of a compact invariant set $K$ can be used to deduce the existence of solutions of $\dot x= f(t,x)$.
	Assume the homology Conley index $\Hom^\IF_*(f_0,K)$ is known and $\Hom^\IF_q(f_0,K)\neq 0$ for some
	$q\in\IZ$. It follows that there are at least $\dim \Hom^\IF_q\Con(f_0, K)$ distinct (full bounded) solutions of
	$\dot x = f(t,x)$ having Morse index $q$ i.e., connecting equilibria having Morse index $q$. 
	
	The computation of the Conley index is in general a non-trivial problem, but there are cases which
	are well understood. Assume that $B\in\mathcal{L}(E,E)$ has no eigenvalues whose real part is zero,
	and let $m_0$ denote the dimension of its generalized eigenspace belonging to the set of eigenvalues
	having positive real parts. If $\eps f(t,\eps^{-1} x) - Bx \to 0$ as $\eps\to 0$ uniformly on $\IR\times B_{1}(0)$,
	then \cite[Theorem 7.16]{article_naci} there exists a largest compact invariant subset $K_\infty$
	in $\Hull^+(f_0)\times E$ and 
	\begin{equation*}
		\Hom^\IF_q\Con(f_0,K_\infty) \iso \begin{cases}
			\IF & q=m_0\\
			0 & q\neq m_0
		\end{cases}
	\end{equation*}
	Therefore, there must exist at least one solution $u:\IR\to E$ of $u_t = f(t,x)$ converging as $t\to\pm\infty$
	to equilibria $e^{\pm}$ having Morse-index $m_0$.
	
	Mutatis mutandis the complete introduction applies not only to ordinary differential equations
	but to a large class of abstract semilinear parabolic equations. Narrowing down the focus to
	reaction-diffusion equations on smooth bounded domains, we can hide most of the complexity and
	formulate a relatively simple result. 
	
	Let $\Omega \subset \IR^{m_0}$, $m_0\geq 1$ be a bounded domain with smooth boundary. Assume\footnote{Strictly speaking, it is not required to assume that $f$ is continuously differentiable in $t$.} 
	\begin{enumerate}
		\item[(A1)]
		$f\in C^k(\IR\times\Omega\times\IR, \IR)$ is asymptotically autonomous meaning that there are $f^{\pm\infty}\in C^k(\Omega\times \IR,\IR)$ such that 
		\begin{equation*}
		\Diff^l \left(f(t,x,u) - f^{\pm\infty}(x,u)\right)\to 0
		\end{equation*} 
		for all $l\in\{1,\dots,k\}$ and uniformly on sets of the form $\Omega\times B_\eps(0)$, $\eps>0$.
	\end{enumerate}

	Take $X:=L^p(\Omega)$ for some $p\geq m_0$, and define an operator (using Dirichlet boundary conditions)
	\begin{align*}
		A &: W^{2,p}(\Omega)\cap W^{1,p}_0(\Omega) \to L^p(\Omega)\\
		Au &:= -\Delta u
	\end{align*}
	$A$ is a positive sectorial operator having compact resolvent. Let $a \in \rho(A)$, where $\rho(A)$ denotes the resolvent set of $A$, and suppose that 
	\begin{enumerate}
		\item[(A2)]
			\begin{equation*}
			\eps f(t,x,u) - au \to 0 \text{ as }\eps\to 0
			\end{equation*}
			uniformly on sets of the form $\IR\times\Omega\times B_\eps(0)$. The dimension
			of the eigenspace belonging to the negative part of the spectrum of $A-a$
			is referred to as the Morse-index $m(A-a)$ of $A-a$.
	\end{enumerate}
	
	Let $\mathcal{F}$ denote the set of all $f\in C^{k}(\IR\times\Omega\times\IR, \IR)$, $k\geq 1$
	satisfying (A1) and (A2). We treat $\mathcal{F}$ as a subspace of $C^k(\IR\times\Omega\times\IR,\IR)$.
	The topology is induced by a family $(\delta_{n,k})$ of seminorms:
	\begin{align*}
		\delta_{n,0}(f) &:= \sup_{(t,x,y) \in \IR\times \Omega\times B_n(0; \IR)} \abs{f(t,x,y)}\\
		\delta_{n,l}(f) &:= \sup_{(t,x,y) \in \IR\times \Omega\times B_n(0; \IR)} \abs{\Diff^k f(t,x,y)}\quad l\in\{1,\dots, k\}
	\end{align*}
	A neighbourhood sub-basis for an element $f\in C^k(\IR\times \Omega\times\IR,\IR)$ 
	is formed by sets of the form
		\begin{equation*}
		N_{n,l,\eps} := \{ g\in C^k(\IR\times \Omega\times\IR, \IR):\; \delta_{n,l}(f-g)<\eps\}\quad (n,l,\eps)\in\IN\times\{0,\dots,k\}\times\IR^+
	\end{equation*}

	Let $f\in\mathcal{F}$ be arbitrary. We write $f^{\pm\infty}$ to denote the asymptotic time limit given by (A1).
	$\hat f(t,u)(x) := f(t,x,u(t,x))$ (resp. $\hat f(u) := f(x,u)$) is the Nemitskii operator induced by $f$. Note that $\hat f(t,u)\in L^p(\Omega)$ for all $(t,u)\in \IR\times L^p(\Omega)$ (resp. $\hat f(u)\in L^p(\Omega)$ for all $u\in L^p(\Omega)$) due to (A2).
	
	\begin{theorem}
		\label{th:171208-1449}
		For a generic $f\in\mathcal{F}$ 
		there exists a solution $u:\IR\to W^{1,p}_0(\Omega)$ of 
		\begin{equation*}
			\dot u + Au = \hat f(t,u)
		\end{equation*}
		such that:
		\begin{enumerate}
			\item there are, of course, equilibria $e^\pm$ of $\dot u + Au = \hat f^{\pm\infty}(u)$ of Morse-index $m(A-a)$
			\item with $u(t)\to e^\pm$ as $t\to\pm\infty$;
			\item $u$ is hyperbolic i.e., $v_t + Av = \hat f_u(t,u(t))v$ admits an exponential dichotomy.
		\end{enumerate}
	\end{theorem}
	
	It follows from \cite{brunpol} that there exists a residual subset $\mathcal{F}_0$ of $\mathcal{F}$
	such that for every $f\in\mathcal{F}_0$, all equilibria of the limit equations given
	by $f^{\pm\infty}$ are hyperbolic, and their stable and unstable manifolds intersect transversely
	We can now use Theorem \ref{th:171213-1457} to prove that the conclusions of Theorem \ref{th:171208-1449}
	hold for a generic $f\in \mathcal{F}_0$.

	Let us conclude with a brief overview of the following sections. After a Preliminaries section,
	we give	a short review of basic concepts and definitions of nonautonomous Conley index theory, which
	is not widely known and crucial for many proofs in this article.
	In Section \ref{sec:parameter}, we discuss appropriate choices for the parameter space $Y$
	associated with the skew product formulation of the problem. Subsequently, we construct
	an example of an equilibrium-like chain-recurrent set having trivial index and prove a uniformity
	theorem leading to \eqref{eq:171206-1429}. Section \ref{sec:abstracteq} is devoted to the abstract
	semilinear parabolic setting and the proof of an existence result for hyperbolic solutions.
	
	\begin{section}{Preliminaries}
		Almost all spaces in this paper are metric. We write $d(.,.)$ to denote
		the metric without explicitly referencing the space as long as the meaning is clear. Given
		two metric spaces $(X_1, d_1)$	and $(X_2, d_2)$ are metric spaces, their intersection $X_1\cup X_2$ is considered to be another metric space
		equipped with the metric defined by $d(x,y) := d_1(x,y) + d_2(x,y)$. 
		
		The open ball with radius $\eps$ and centre $x_0$ in a metric space $(X,d)$ 
		is denoted by $B_\eps(x;X)$ and the closed ball with the same radius and centre
		by $B_\eps[x;X]$. The notation of $X$ can be omitted when the choice of $X$
		can be deduced from the context.
		
		Let $A$ be a positive sectorial operator on a Banach space $W$, and define the 
		fractional power spaces $W^\alpha$ as the range of $A^{-\alpha}$. Fix some $\alpha\in\left]0,1\right[
		$, and let $f\in C(\IR\times W^\alpha, W)$
		be continuously differentiable in its second variable. A bounded (mild) solution
		$u:\;\IR\to X^\alpha$ of
		\begin{equation*}
			\dot u + Au = f(t,u)
		\end{equation*}
		is called {\em weakly hyperbolic} if the linearisation
		\begin{equation}
			\label{eq:171213-1650}
			\dot v + Av  = \Diff f(t,u(t))v
		\end{equation}
		does not admit a bounded solution $v:\;\IR\to W^\alpha$. A more abstract
		approach to weak hyperbolicity can be found in Section \ref{sec:simple_case} or in \cite{article_index}.
		$u$ is called {\em hyperbolic} if the evolution operator defined by \eqref{eq:171213-1650} 
		admits an exponential dichotomy over $\IR$. 
	\end{section}

	\begin{section}{A review of nonautonomous Conley index theory}
		We assume that the reader is basically familiar with the Conley index
		as can be found in Conley's monograph \cite{conley1978isolated}. If one is
		especially interested in applications to the theory of partial differential equations,
		\cite{ryb} should also be a good starting point.

		There are mainly two intertwined concepts one encounters when dealing
		with nonautonomous problems: processes, sometimes also called evolution operators,
		or skew product semiflows.

		\begin{subsection}{Evolution operators and skew product semiflows}
	 		Let $X$ be a metric space. Assuming that $\undefined\not\in X$, we
	 		introduce a symbol $\undefined$, which means "undefined". The intention is
	 		to avoid the distinction if an evolution operator is defined for a given argument
	 		or not. Define $\overline A := A\dot\cup\{\undefined\}$ whenever $A$ is a set with $\undefined\not\in A$.
	 		Note that $\overline{A}$ is merely a set, the notation does not contain any
	 		implicit assumption on the topology. 
				 		
	 		\begin{definition}
	 			\label{df:131009-1504}
	 			Let $\Delta:=\{(t,t_0)\in\IR^+\times\IR^+:\; t\geq t_0\}$.
	 			A mapping $\Phi:\; \Delta\times \overline{X}\to \overline{X}$
	 			is called an {\em evolution operator} if
				\begin{enumerate}
	 				\item $\Def(\Phi) := \{(t,t_0,x)\in\Delta\times X:\; \Phi(t,t_0,x)\neq\undefined\}$ is open
		 				in $\IR^+\times\IR^+\times X$;
					\item $\Phi$ is continuous on $\Def(\Phi)$;
					\item $\Phi(t_0,t_0, x) = x$ for all $(t_0,x)\in\IR^+\times X$;
					\item $\Phi(t_2 ,t_0,x) = \Phi(t_2,t_1,\Phi(t_1,t_0,x))$ for all $t_0\leq t_1\leq t_2$ in $\IR^+$ and $x\in X$;
				 	\item $\Phi(t,t_0,\undefined) = \undefined$ for all $t\geq t_0$ in $\IR^+$.
				 \end{enumerate}
				 			
	 			A mapping $\pi:\; \IR^+\times \overline{X}\to \overline{X}$ is called {\em semiflow}
	 			if $\tilde\Phi(t+t_0,t_0,x) := \pi(t,x)$ defines an evolution operator.
	 			To every evolution operator $\Phi$, there is an associated
	 			(skew product) semiflow $\pi$ on an extended phase space $\IR^+\times X$, defined by $(t_0,x)\pi t = (t_0+t, \Phi(t+t_0,t_0,x))$.
				 			
				A function $u:\; I\to X$ defined on a subinterval $I$ of $\IR$ is called a {\em solution of (with respect to) $\Phi$}
				if $u(t_1) = \Phi(t_1,t_0,u(t_0))$ for all $\left[t_0,t_1\right]\subset I$.
			\end{definition}
			
			Summing up, if $\Phi$ depends on $t-t_0$ instead of $t$ and $t_0$ it can be seen as a semiflow. Up to the symbol $\undefined$, one can easily derive the usual properties
			of a local semiflow.		
			
			Having a notion of semiflows and solutions of semiflows, one can define invariant sets
			respectively their half-sided equivalents.
			Note that we do not employ genuinely nonautonomous notions of invariance such as
			pullback attractors or the like. In this paper, an invariant set always refers to a semiflow, albeit usually
			a skew product semiflow.
			
	 		\begin{definition}
	 			Let $X$ be a metric space, $N\subset X$ and $\pi$ a semiflow on $X$. The set
	 			\begin{equation*}
		 			\Inv^-_\pi(N) := \{x\in N:\;\text{ there is a solution }u:\;\IR^-\to N\text{ with }u(0) = x\}
	 			\end{equation*}
	 			is called the {\em largest negatively invariant subset of $N$}.
				 			
	 			The set
	 			\begin{equation*}
		 			\Inv^+_\pi(N) := \{x\in N:\;x\pi\IR^+\subset N\}
				\end{equation*}
				is called the {\em largest positively invariant subset of $N$}.
				 			
				The set
				\begin{equation*}
		 			\Inv_\pi(N) := \{x\in N:\;\text{ there is a solution }u:\;\IR\to N\text{ with }u(0) = x\}
				\end{equation*}
				is called the {\em largest invariant subset of $N$}.
			\end{definition}
			
			In the skew product setting as opposed to the process setting, a metric
			space $Y$ is considered in addition to $X$. A skew product semiflow on $Y\times X$
			is based on a semiflow on $Y$. 
			
			Assume
			that $(t,y)\mapsto y^t$ is a semiflow on $Y$, sometimes called $t$-translation.
			The (positive) hull denotes the closure of an orbit (semiorbit) in
			the phase space. With respect to the space $Y$ and the semiflow obtained
			by translation, we have
	 		\begin{definition}
	 			For $y\in Y$ let
	 			\begin{enumerate}
	 				\item \begin{equation*}
	 				\Hull^+(y) := \cl_Y \{y^t:\;t\in\IR^+\} 
	 				\end{equation*}
	 				denote the positive hull of $y$;
	 				\item \begin{equation*}
	 				\Hull(y) := \cl_Y \{y^t:\;t\in\IR\} 
	 				\end{equation*}
	 				denote the hull of $y$. 
	 			\end{enumerate}
	 		\end{definition}
				 		
			The skew product semiflows used in this papers always refer
			to the $t$-translation as semiflow on the first component.				
			\begin{definition}
					We say that $\pi = (.^t,\Phi)$ is a skew product semiflow on $Y\times X$
					if $\Phi:\; \IR^+\times \overline{Y\times X}\to \overline{Y\times X}$ is a mapping such that 
					\begin{equation}
						\label{eq:171129-1749}
						(t,y,x)\pi t := 
						\begin{cases}
			 			(y^t, \Phi(t,y,x)) & \Phi(t,y,x)\neq\undefined\\
			 			\undefined & \text{otherwise}
			 			\end{cases}
					\end{equation}
					is a semiflow on $Y\times X$.
			\end{definition}
	
			It is a standing assumption that $X$ and $Y$ are metric spaces,
			and $\pi$ is a skew product semiflow as defined above. Evolution
			operators are easily recovered from the skew product semiflow 
			by setting $\Phi_y(t,t_0,x) := \Phi(t-t_0,y^{t_0}, x)$ for $t\geq t_0$
			and $x\in X$.			 		
			
			Conversely, every process can be understood as semiflow on
			the space $\IR^+\times X$. In case the process is given by $\Phi_{y_0}$,
			this semiflow is denoted by $\chi :=\chi_{y_0}$ and $(t,x)\chi_{y_0} s:= \Phi_{y_0}(t+s,t,x)$.

			Conley index theories can be read as theories about isolating
			neighbourhoods even to the extent that one replaces invariant sets
			by isolating neighbourhoods. 
			
			In contrast to the autonomous case, the notion of an isolating neighbourhood
			is less obvious. The definition given below refers also to the space $Y\times X$
			instead of $\Hull^+(y_0)\times X$, which is useful for continuation arguments.
			
			\begin{definition}
				\label{df:isolating-neighbourhood-1}
				Let $y_0\in Y$ and $K\subset \Hull^+(y_0)\times X$ be an invariant set. 
				A closed set $N\subset Y\times X$ (resp. $N\subset \Hull^+(y_0)\times X$)
				is called an {\em isolating neighbourhood} for $(y_0, K)$ (in $Y\times X$) (resp. in $\Hull^+(y_0)\times X$)
				provided that:
				\begin{enumerate}
					\item $K\subset \Hull^+(y_0) \times X$
					\item $K\subset \interior_{Y\times X} N$ (resp. $K\subset \interior_{\Hull^+(y_0)\times X} N$)
					\item $K$ is the largest invariant subset of $N\cap (\Hull^+(y_0)\times X)$
				\end{enumerate}				 	 							
			\end{definition}				
		\end{subsection}	
		
		\begin{subsection}{Index pairs and the homotopy index}
			The Conley index can be obtained as follows. Take an index pair
			for an appropriate isolated invariant set and collapse the exit
			set to one point. The homotopy type of the resulting topological
			space is the homotopy index.
			
	 		\begin{definition}
	 			\label{df:140606-1652}
	 			Let $X$ be a topological space, and $A,B\subset X$. Denote
	 			\begin{equation*}
		 			A/B := A/R \cup \{A\cap B\},
				\end{equation*}
				where $A/R$ is the set of equivalence classes
				with respect to the relation $R$ on $A$ which is defined by $x R y$ iff $x=y$ or $x,y\in B$.
				 			
				We consider $A/B$ as a topological space
				endowed with the quotient topology with respect to the
				canonical projection $q:\; A\to A/B$, that is, a
				set $U\subset A/B$ is open if and only if
				\begin{equation*}
					q^{-1}(U) = \bigcup_{x\in U} x
				\end{equation*}
				is open in $A$.
			\end{definition}
			
			Now the term "collapsing the exit set" can be made precise, but
			we still do not know what an index pair is supposed to be.
			
		 	\begin{definition}
		 		\label{df:basic_index_pair}
		 		A pair $(N_1, N_2)$ is called a {\em (basic) index pair} relative to a
		 		semiflow $\chi$ in $\IR^+\times X$ if
		 		\begin{enumerate}
		 			\item[(IP1)] $N_2\subset N_1\subset \IR^+\times X$, $N_1$ and $N_2$ are
			 			closed in $\IR^+\times X$
		 			\item[(IP2)] If $x\in N_1$ and $x\chi t\not\in N_1$ for some $t\in\IR^+$, 
			 			then $x\chi s\in N_2$ for some $s\in\left[0,t\right]$;
		 			\item[(IP3)] If $x\in N_2$ and $x\chi t\not\in N_2$ for some $t\in\IR^+$, then
			 			$x\chi s\in (\IR^+ \times X)\setminus N_1$ for some $s\in\left[0,t\right]$.
		 		\end{enumerate}
		 	\end{definition}
		 	
			Index pairs need to be associated with invariant sets in order to define
			an index. We have already defined isolating neighbourhoods, but appropriate
			compactness assumptions is still missing. This is a non-problem in the finite-dimensional
			case, but crucial in an infinite-dimensional setting.
			
	 		\begin{definition}
				A closed set $M\subset Y\times X$ is called {\em strongly admissible} (or {\em asymptotically compact})
				provided the following holds:
				 			
	 			Whenever $(y_n,x_n)$ is a sequence in $M$ and $(t_n)_n$ is a sequence
	 			in $\IR^+$ such that $(y_n,x_n)\pi\left[0,t_n\right]\subset M$, then
	 			the sequence $(y_n,x_n)\pi t_n$ has a convergent subsequence.
	 		\end{definition}
	 		
	 		\begin{definition}
	 			A closed set $M\subset Y\times X$ is called {\em skew-admissible} provided
	 			the following holds: whenever $(y_n, x_n)_n$ in $N$ and $(t_n)_n$
	 			in $\IR^+$ are sequences such that $t_n\to\infty$, $y^{t_n}_n\to y_0$ in $Y$
	 			and $(y_n,x_n)\pi \left[0,t_n\right]\subset N$, the sequence $\Phi(t_n,y_n,x_n)$
	 			has a convergent subsequence.
	 		\end{definition}
	 		
	 		Suppose $y_0\in Y$ is a so-called initial element and $K\subset \Hull^+(y_0)\times X$
	 		an invariant set admitting a strongly admissible isolating neighbourhood. Then, there
	 		exists an index pair $(N_1, N_2)$ for $(y_0, K)$ as defined below. Moreover, the homotopy
	 		type of the pointed space $(N_1/N_2, N_2)$ is independent of the choice of an index pair. 
	 		Hence, the homotopy index is well-defined.
	 		
	 		\begin{definition}
	 			\label{df:140128-1503}
	 			Let $y_0\in Y$ and $(N_1, N_2)$ be a basic index pair in $\IR^+\times X$ relative to $\chi_{y_0}$. Define
	 			$r:=r_{y_0}:\; \IR^+\times X\to \Hull^+(y_0)\times X$ by $r_{y_0}(t,x) := (y^t_0,x)$.
	 			 		
	 			Let $K\subset \omega(y_0)\times X$ be an (isolated) invariant
	 			set. We say that $(N_1, N_2)$ is a (strongly admissible) index pair\footnote{Every index pair in the
	 			sense of Definition \ref{df:140128-1503} is assumed to be strongly admissible.} for $(y_0, K)$ if:
	 			\begin{enumerate}
	 				\item[(IP4)] there is a strongly admissible isolating neighbourhood $N$ of $K$ in $\Hull^+(y_0)\times X$
	 					such that $N_1\setminus N_2 \subset r^{-1}(N)$;
	 				\item[(IP5)] there is a neighbourhood $W$ of $K$ in $\Hull^+(y_0)\times X$
	 					such that $r^{-1}(W)\subset N_1\setminus N_2$.
	 			\end{enumerate}
	 		\end{definition}
	 		
	 		A frequently used construction enlarges the exit set of a given index pair $(N_1, N_2)$.
 		 	\begin{definition}
 		 		Let $(N_1, N_2)$ be an index pair in $\IR^+\times X$ (relative to the semiflow
		 		$\chi$ on $\IR^+\times X$). For $T\in\IR^+$, we set
 		 		\begin{equation*}
	 		 		N^{-T}_2 := N^{-T}_2(N_1) := \{(t,x)\in N_1:\; \exists s\leq T\; (t,x)\chi s\in N_2\}.
 		 		\end{equation*}
 		 	\end{definition}
		\end{subsection}		
			 		
		\begin{subsection}{The categorial Conley index and attractor-repeller decompositions}
  			A {\em connected simple system} is a small category such that given
  			a pair $(A,B)$ of objects, there is exactly one morphism $A\to B$.
  			
  			Let $y_0\in Y_c$ and $K\subset \Hull^+(y_0)\times X$
  			be an isolated invariant set for which there is a strongly admissible
  			isolating neighbourhood.	The categorial Conley index $\Con(y_0, K)$ (as defined in \cite{article_naci_1}) is a
  			subcategory of the homotopy category of pointed spaces and a connected simple system. 
  			Its objects are pointed spaces $(N_1/N_2, N_2)$, where $(N_1, N_2)$ is an index pair for $(y_0, K)$. Roughly speaking, one can think
  			of an index pair with collapsed exit set as a representative of the index. All of the
  			representatives are isomorphic in the homotopy category of pointed spaces.
  			
  			Let $(\Hom_*,\partial) := (\Hom^{\Gamma}_*, \partial)$ denote a homology theory with compact supports \cite{spanier}
  			and coefficients in a module $\Gamma$. In this paper, we often assume that $\Gamma = \IF$,
  			where the latter denotes an arbitrary field.
  			Recall that $\Hom_*$ is a covariant functor from the category of topological
  			pairs to the category of graded abelian groups (or modules). 
  			
  			Define the {\em homology Conley index} $\Hom_*\Con(y_0,K)$ to be the following
  			connected simple system: $\Hom_*(N_1/N_2, \{N_2\})$ is an object
  			whenever
  			$(N_1/N_2, N_2)$ is an object of $\Con(y_0, K)$. The morphisms of $\Hom_*\Con(y_0, K)$
  			are obtained analogously from the morphisms of $\Con(y_0, K)$. 
  			Note that we also write $\Hom_*(A,a_0) := \Hom_*(A,\{a_0\})$ provided the meaning is clear.
  			
  			Previously, we have defined invariance relying only on the skew
  			product formulation. The same approach will be used for attractor-repeller
  			decomposition. First of all given a solution $u:\;\IR\to \Hull^+(y_0)\times X$ of $\pi$,
  			its $\alpha$-- and $\omega$-limes sets are defined as usual.
	 		\begin{align*}
  		 		\alpha(u) &:= \bigcap_{t\in\IR^-} \cl_{\Hull^+(y_0)\times X} u(\left]-\infty,t\right])\\
  		 		\omega(u) &:= \bigcap_{t\in\IR^+} \cl_{\Hull^+(y_0)\times X} u(\left[t,\infty\right[)
  	 		\end{align*}
  				 		
  	 		Based on these definitions, the notion of an attractor-repeller decomposition
  	 		can be made precise.
  				 		
  	 		\begin{definition}
  	 			Let $y_0\in Y$ and $K\subset \Hull^+(y_0)\times X$ be an isolated
  	 			invariant set. $(A,R)$ is an {\em attractor-repeller decomposition} of 
  	 			$K$ if $A,R$ are disjoint isolated invariant subsets of $K$ and for every solution $u:\;\IR\to K$
  	 			one of the following alternatives holds true.
  				\begin{enumerate}
  	 				\item $u(\IR)\subset A$
  	 				\item $u(\IR)\subset R$
  	 				\item $\alpha(u)\subset R$ and $\omega(u)\subset A$
  	 			\end{enumerate}
  	 			We also say that $(y_0, K, A, R)$ is an attractor-repeller decomposition.
  	 		\end{definition}
  				 		
  			Given an attractor-repeller decomposition $(y_0, K, A, R)$, there is a long exact \cite{article_naci_1}
  			sequence
  			\begin{equation*}
  			\xymatrix@1{
  				\ar[r] & \Hom_*\Con(y_0,A) \ar[r] & \Hom_*\Con(y_0,K) \ar[r]
  				& \Hom_*\Con(y_0,R) \ar[r]^-{\partial} & \Hom_{*-1}\Con(y_0,A)
  				\ar[r]&
  			}
  			\end{equation*}
  			where $\partial$ denotes the connecting homomorphism. The above sequence is called
  			the {\em attractor-repeller sequence}. 
		\end{subsection}
					
	\end{section}

	\begin{section}{The parameter space}
		\label{sec:parameter}
		With respect to (nonautonomous) Conley index calculations, a fixed
		parameter space $Y$ is used throughout this paper. Suppose that $E_0$
		is a metric vector space endowed with an invariant metric $d$ i.e., $d(x,y) = d(x-y,0)$ for all $x,y\in E_0$.
		Let $(E^i_0)_{i\in I}$ be a family of subsets such that 
		\begin{equation*}
			\bigcup_{i\in I} E^i_0 = E_0
		\end{equation*}
		
		Define $Y$ to be the subspace of $C(\IR, E_0)$ consisting of all $y:\;\IR\to E_0$ such that
		$y$ is continuous and $y(\IR)\subset E^i_0$ for some $i\in I$. $Y$ is endowed
		with a metric $d:=d_Y$ generated by a family $(\delta_n)_{n\in\IN}$ of seminorms given by
		\begin{equation*}
			\delta_n(f) := \delta_n(f) := \sup_{\abs{t}\leq n} \norm{f(t)}_{E_0}
		\end{equation*}
		and 
		\begin{equation}
			\label{eq:171129-1327}
			d_Y(y,y') := d_Y(y-y',0) := \sum^\infty_{n=1} 2^{-n} \frac{\delta_n(y)}{1+\delta_n(y)}
		\end{equation}
		The metric $d_Y$ induces the compact open topology on $Y$. None of
		the proofs relies on the particular choice of a metric $d_Y$ but only on the topology. 
		
		The translation $(t,y)\mapsto y^t$ on $\IR\times Y$ is defined by 
		$y^t(s) := y(t+s)$. 
		
		Before dealing with the particular problem of asymptotically autonomous equations,
		we will give a short overview of possible realizations of $E_0$ and $Y$ in the case ordinary differential
		equations and semilinear parabolic equations.
		
		\begin{subsection}{Ordinary differential equations}
			Let $E$ be a finite-dimensional real normed space, e.g. $E=\IR^N$ for some $N\in\IN$.
			By an ordinary differential equation we mean a differential equation
			\begin{equation*}
				\dot x = f(t,x)
			\end{equation*}
			where $f$ is assumed to be continuous in $t$ and Lipschitz-continuous in $x$
			uniformly on bounded sets.
			
			Let $E_0$ denote the subspace of all $g\in C(E, E)$ which are Lipschitz-continuous
			uniformly on bounded sets equipped with a metric induced by a family $(\delta_n)_{n\in\IN}$
			of seminorms given by
			\begin{equation*}
				\delta_n(g) := \sup_{\abs{x}\leq n} \norm{g(x)}_{E}
			\end{equation*}

			Let $I$ denote the set of all $L:\; \IR^+ \to \IR^+$ and 
			\begin{equation*}
				E_L := \{g\in E_0 :\; \norm{g(x) - g(x')} \leq L(C)\norm{x-x'}\text{ for all } x,x'\text{ with }\norm{x},\norm{x'}\leq C\}.
			\end{equation*}			
			
			We have $y\in Y$ if and only if the mapping $f:\;\IR\times E\to E$, $f(t,x) := y(t)(x)$,
			is continuous and Lipschitz-continuous uniformly on sets of the form $\IR\times B$ where $B\subset E$ is 
			an arbitrary bounded set.
			
			The skew product semiflow $\pi$ is determined by its cocycle $\Phi_y$
			and \eqref{eq:171129-1749}, using $X:=E$ as phase space. Let $\left[t_0,t\right]\subset \IR$
			be an interval. We set $\Phi_y(t,t_0,x_0) = x_1$ if $u:\;\left[t_0,t\right]\to E$
			is a solution of 
			\begin{equation*}
				\dot x = y(t)(x)
			\end{equation*}
			with $u(t_0) = x_0$ and $u(t) = x_1$.
			
			There are two relevant notions of asymptotic compactness. A closed set $N\subset \Hull^+(y_0)\times X$, $y_0\in Y$
			is strongly admissible if $\Hull^+(y_0)$ is compact and $\sup\{\norm{x}:\; (y,x)\in N\}<\infty$.
			$N\subset Y\times X$ is strongly skew-admissible if it is closed and $\sup\{\norm{x}:\; (y,x)\in N\}<\infty$.
			
			In this article, the compactness of $\Hull^+(y_0)$ follows directly from its construction. It
			is not necessary to impose additional compactness conditions on $y_0$ like e.g. uniform continuity.
		\end{subsection}
		
		\begin{subsection}{Semilinear parabolic equations}
			\label{sec:parameters_infinite}
			The last remark of the previous section on ordinary differential equations concerning the
			compactness of $\Hull^+(y_0)$ applies
			to semilinear parabolic equations as well and allows us to formulate the results in an 
			abstract manner as opposed to the approach in \cite{article_index}.
			
			Let $W$ be a Banach space, and $A$ be a positive, sectorial operator on $W$ having compact resolvent. 
			$A$	gives rise to an analytic semigroup denoted by $\e^{-At}$. Let $W^\alpha$ denote
			the $\alpha$-th fractional power space with respect to $A$, equipped with the norm
			$\norm{x}_\alpha := \norm{A^\alpha x}$ (see e.g. \cite{sellyou}, \cite{henry}, \cite{pazy}).
			
			Fix an $\alpha\in\left]0,1\right[$ and a phase space $E:=W^\alpha$. Let $E_0 \subset C(E,W)$
			denote the subspace of all continuous functions mapping bounded sets of $E$ into bounded sets of $W$.
			$E_0$ is equipped with a metric induced by a family $(\delta_n)_n$ of seminorms given by
			\begin{equation*}
				\delta_n(g) := \sup\{\norm{g(x)}:\; x\in E\text{ with } \norm{x}_\alpha\leq n\}
			\end{equation*}
			
			Let $I$ denote the set of all $L:\; \IR^+ \to \IR^+$ and 
			\begin{equation*}
				\begin{split}
				E_L := \{g\in E_0 :\; &\norm{g(x) - g(x')} \leq L(C)\norm{x-x'}_\alpha \\
				&\text{ for all } x,x'\in E\text{ with }\norm{x}_\alpha,\norm{x'}_\alpha\leq C\}
				\end{split}
			\end{equation*}			
					
			We have $y\in Y$ if and only if the mapping $f:\;\IR\times E\to W$, $f(t,x) := y(t)(x)$,
			is continuous and Lipschitz-continuous uniformly on sets of the form $\IR\times B$ where $B\subset E$ is 
			an arbitrary bounded set.
			
			For $(y,x)\in Y\times E$ let $t\mapsto \Phi_y(t,t_0,x)$, $t\geq t_0$ denote the maximally defined
			mild solution of 
			\begin{equation*}
				x_t + Ax = y(t)(x)
			\end{equation*}
			It follows from \cite[Theorem 47.5]{sellyou} that $\pi$ defined by \eqref{eq:171129-1749} 
			using $X := E$ is indeed a semiflow.
			
			Furthermore, a closed set $N\subset Y\times E$ is strongly skew-admissible if 
			\begin{equation*}
				\sup \{\norm{x}_\alpha:\; (y,x)\in N\} < \infty
			\end{equation*}
		\end{subsection}
		
		\begin{subsection}{An initial parameter $y_0$ for asymptotically autonomous problems}
			We say that a parameter $f\in Y$ is autonomous if $f^t=f$ for all $t\in\IR$
			and asymptotically autonomous if $f^t \to f^{\pm\infty}$ as $t\to\pm\infty$
			for some autonomous $f^{\pm\infty}\in Y$.
			
			We do not require that $f^{-\infty} = f^\infty$. Let $(f,g)\in Y\times Y$ be
			a pair of asymptotically autonomous parameters with $f^t\to f^{\pm\infty}$
			and $g^t\to f^{\mp\infty}$ as $t\to\pm\infty$. We also say that $(f,g)$ is
			an {\em asymptotically autonomous cycle}.
			
			Take $C_0\subset \IC$ to be the annulus	
			\begin{equation*}
				C_0 := \left\{r \mathrm{e}^{\imag\varphi}:\; (r,\varphi)\in \left[1/2, 1\right]\times \IR\right\}
			\end{equation*}
			We will consider a fixed initial element $z_0 = \frac{1}{2}$. The differential equation
			\begin{align}
				\label{eq:171215-1809a}
				\dot r &= 1-r\\
				\label{eq:171215-1809b}
				\dot \varphi &= (1-r) + r\cos(\varphi)
			\end{align}
			defines a semiflow on $C_0$. The semiflow is denoted by $z^t$ for $z\in C_0$ and referred
			to as the usual translation. It is easy to see that $\omega(c_0) = \{z\in \IC:\; \abs{z}=1\}$ i.e.,
			the complex unit sphere.
			
			A mapping $F:\;C_0\to E_0$ is obtained by setting
			\begin{equation*}
				F(r e^{\imag\varphi}) := \begin{cases}
					f^{-\infty}(0) & \varphi = -\frac{\pi}{2}\\
					f(t) & \varphi = \Phi(t;0) \\
					f^\infty(0) & \varphi = \frac{\pi}{2}\\
					g(t) & \varphi = \Phi(t;\pi) \\
				\end{cases}
			\end{equation*}
			where $\Phi(t;\varphi_0)$ denotes the solution of $\dot\varphi = \cos(\varphi)$ with 
			$\Phi(0;\varphi_0) = \varphi_0$.
			
			We have $\omega(F(\omega(z_0))) = \Hull(f)\cup \Hull(g)$. The next step
			is to modify $F$ such that it is constant in a small sector around
			the angles $-\pi/2$ and $\pi/2$.
			
			\begin{equation*}
				F^\rho(r e^{\imag \varphi}) := \begin{cases}
					F(-\pi/2) & \abs{\varphi + \pi/2} \leq \rho\\
					F\left(\varphi \frac{\pi}{\pi-2\rho}\right) & -\pi/2+\rho < \varphi < \pi/2-\rho\\
					F(\pi/2) & \abs{\varphi - \pi/2} \leq \rho\\
					F\left(\pi + (\varphi-\pi) \frac{\pi}{\pi-2\rho}\right) & \pi/2+\rho < \varphi < 3\pi/2-\rho\\
				\end{cases}
			\end{equation*}
			Note that $F^0 = F$.
			
			Define $\hat F^\rho:\;C_0\to Y$ by $\hat F^\rho(z)(t) := F^\rho(z^t)$. We have for all $s\in\IR$
			$\hat F^\rho(z)^t(s) = \hat F^\rho(z)(t+s) = \hat F^\rho(z^{t+s}) = \hat F^\rho(z^t)(s)$,
			so $\hat F^\rho(z)^t = \hat F^\rho(z^t)$ for all $z\in C_0$ and all $t\in\IR$. The initial
			element $y^\rho_0 := \hat F^\rho(z_0)$ depends on $f$ and $g$ and is thus also denoted by $y^\rho_0(f,g)$.
			We also write $y_0$ or $y_0(f,g)$ instead of $y^\rho_0$ or $y^\rho_0(f,g)$ if $\rho=0$.
			
			The following lemma is formulated for the categorial index simply because this leads to a
			general statement and implies equality of the homotopy index, the homology index and other
			possibly derived indices.
			\begin{lemma}
				\label{le:171115-1506}
				Let $(f,g)$ be an asymptotically autonomous cycle and
				$N\subset Y\times X$ a strongly skew-admissible isolating neighbourhood
				for $K\subset \Hull^+(y_0)\times X$.
				
				Then for all $\rho>0$ sufficiently small:
				\begin{enumerate}
					\item \label{item:171114-1} $N$ is an isolating neighbourhood for $K_\rho := \Inv(N)\cap (\Hull^+(y^\rho_0)\times X)$.
					\item $\Con(y^\rho_0, K) \iso \Con(y_0, K)$
				\end{enumerate}
			\end{lemma}
	
			\begin{proof}
				First of all, we are going to prove that $d(\left(y^\rho_0\right)^t, \Hull^+(y^\rho_0))\to 0$
				as $\rho\to 0$ and $t\to\infty$. This follows in particular if we show that
				\begin{equation*}
					\sup_{t\in\IR^+} d\Bigl(\bigl(\underbrace{\hat F^\rho(z_0)}_{=y^\rho_0}\bigr)^t, \bigl(\hat F^0(z_0)\bigr)^t\Bigr) \to 0
				\end{equation*}
				as $\rho\to 0$. We have
				\begin{equation*}
					\left(\hat F^\rho(z_0)\right)^t = \hat F^\rho(z^t_0)
				\end{equation*}
				so it is sufficient to note that $F^\rho\to F^0$ as $\rho\to 0$ uniformly on $C_0$ which is compact.
				
				\begin{enumerate}
					\item This follows immediately from Theorem 5.6 in \cite{article_naci}.
					\item It follows from (\ref{item:171114-1}) and the remarks at the beginning
					of this proof that $\gamma:\left[0,1\right]\to Y\times 2^{Y\times X}$, $\rho\mapsto (y^\rho_0, K_\rho)$ is continuous in the sense required by Theorem 5.12 in \cite{article_naci}, whence
					it follows that the Conley indices are isomorphic.
				\end{enumerate}
			\end{proof}
		\end{subsection}
	\end{section}	 
	
	\begin{section}{Simple weakly hyperbolic "equilibria"}
		\label{sec:simple_case}
		Hyperbolic equilibria play an important role among all equilibria.
		It is certainly not a new question what hyperbolicity could mean
		in a nonautonomous context. Often the answer depends on the
		context, and therefore exponential dichotomies have proven to 
		be a useful concept.
		
		In the context of Conley index theory, built around the
		notion of isolated invariant sets, it seems unnatural to use growth
		estimates to characterize hyperbolicity. We use instead
		the notion of {\em weak hyperbolicity} which has been introduced
		in \cite{article_index}. A very similar notion of weak hyperbolicity
		has been used earlier in \cite{sacker1994dichotomies}. The results
		in \cite{sacker1994dichotomies} also imply that hyperbolicity and
		weak hyperbolicity are often equivalent.
		
		A parameter $f\in Y$ is called linear if the evolution operator
		$\Phi_f$ is linear. A linear parameter $f$ is called
		{\em weakly hyperbolic} provided every strongly admissible set $N\subset \Hull(f)\times E$
		is invariant. 
		
		Let $u$ be a bounded solution of $\Phi_f$ for some $f\in Y$. Since
		$f(t)$ is continuously differentiable for all $t\in \IR$, we can
		define $d(f,u)\in Y$ by $d(f,u)(t)(x) := D[f(t)](u(t))x$. We say
		that $u$ is a weakly hyperbolic solution of $\Phi_f$ or that $(f,u)$ is weakly hyperbolic
		if $d(f,u)$	is weakly hyperbolic as defined previously in the linear case.
		
		An invariant set $K\subset Y\times X$ is called weakly hyperbolic if
		$(f,u)$ is weakly hyperbolic whenever $u$ is a solution of $\Phi_f$
		such that $(f^t,u(t))\in K$ for all $t\in\IR$.
				
		Invariant sets in our context
		are (trivial) bundles over the nonautonomous nonlinearities. The projection
		associated with the bundle structure will be visualized by vertical arrows.
		\begin{equation}
		\label{eq:171123-1454}
		\xymatrix@C=2cm{
			e^- \ar[d] \ar@/_5mm/[r]^-u & \ar@/_5mm/[l]^-v e^+ \ar[d]\\
			f^{-\infty} \ar@/_5mm/[r]^-{f} & \ar@/_5mm/[l]^-g f^\infty
		}
		\end{equation}
		The above diagram reads as follows:
		$f^t \to f^{\pm\infty}$ as $t\to\pm\infty$ and $g^t\to f^{\mp\infty}$ as $t\to\pm\infty$.
		Furthermore, $u$ (resp. $v$) is a solution of $\dot x = f(t,x)$ (resp. $\dot x = g(t,x)$)
		with $u(t)\to e^\pm$ (resp. $v(t)\to e^\mp$) as $t\to\infty$. Together,
		the diagram describes the invariant set
		\begin{equation*}
		K = \{(f^{-\infty},e^-), (f^\infty, e^+)\} \cup \{(f^t,u(t)):\;t\in\IR\} \cup \{g^t,v(t):\; t\in\IR\}
		\end{equation*}
		
		Suppose $K\subset Y\times E$ is an invariant set having the structure of
		\eqref{eq:171123-1454}. By using the techniques of Section \ref{sec:parameter},
		we can obtain an initial element $y_0\in Y$ such that $\omega(y_0) = \Hull(f)\cup \Hull(g)$
		as required by $K$. Therefore $\Con(y_0, K)$ and consequently $\Hom_*\Con(y_0, K)$ are
		defined. 
		
		For every $(f,u)\in K$, we set $(f\ominus u)(t)(x) := f(t)(x + u(t))$, so
		\begin{equation*}
			K_0 := \{(f\ominus u,0) :\; (f,u)\in K\}
		\end{equation*}
		and $K$ are related by a homeomorphism.
		
		Therefore, we can assume without loss of generality that $K=K_0$,
		so Theorem 7.13 in \cite{article_index} is applicable, whence it follows 
		that form some $m\in \IZ$
		\begin{equation*}
			\Hom^\IF_q\Con(y_0, K) \iso \begin{cases} \IF & q = m\\ 0 & q\neq m \end{cases}
		\end{equation*}
		
		At the moment, we can only guess that $m$ must agree with the Morse
		indices of the equilibria $e^+$ and $e^-$. 
		
		We will not elaborate on 
		this as it follows immediately from Theorem \ref{th:171115-1441} proved below
		and formulas for the (autonomous) Conley index of a hyperbolic equilibrium
		(see e.g. Theorem 11.1 in \cite{ryb}).
	\end{section}
	
	\begin{section}{"Equilibria" with trivial index}
		\label{sec:trivial}
		There is a simple formula \eqref{eq:171122-1341} 
		for the Conley index of an isolated equilibrium. Replacing equilibria
		by finest Morse sets sets as explained in the introduction, the homology Conley index need
		no longer be non-trivial. This will be shown by constructing an example
		of a non-trivial, weakly hyperbolic chain-recurrent set with trivial homology index.
	
		Let us begin with a remark on the minimum
		complexity required to produce such a "silent" solution. The following
		invariant set certainly has a trivial homology index because the uniformity 
		property of the nonautonomous Conley index for a non-trivial index would imply
		the existence of a full solution $v:\;\IR\to E$ of $\dot x = g(t,x)$.
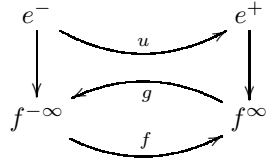
		\begin{equation*}
		\xymatrix@C=2cm{
			e^- \ar[d] \ar@/_5mm/[r]^-u &  e^+ \ar[d]\\
			f^{-\infty} \ar@/_5mm/[r]^-{f} & \ar@/_5mm/[l]^-g f^\infty
		}
		\end{equation*}
		
		Our second example constitutes a nonautonomous equivalent of a homoclinic
		solution, namely two connections from $(f^{-\infty}, e^-)$ to $(f^\infty, e^+)$.
		This resembles the double root of a real function.
		
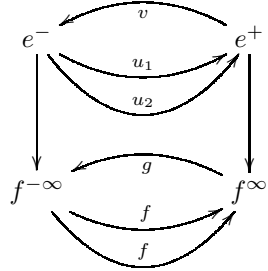
\begin{equation}
		\label{eq:171123-1318}
		\xymatrix@C=2cm@R=15mm{
			e^- \ar[d] \ar@/_5mm/[r]^-{u_1} \ar@/_10mm/[r]^-{u_2}&  \ar@/_5mm/[l]^-v  e^+ \ar[d]\\
			f^{-\infty} \ar@/_5mm/[r]^-{f} \ar@/_10mm/[r]^-{f}& \ar@/_5mm/[l]^-g f^\infty
		}
		\end{equation}

		Taking $E=\IR^2$ and $Y$ as above, $E_0$ denotes the set of all continuously differentiable
		mappings $\IR^2\to \IR^2$ equipped with the maximum norm. We can identify an element $y\in Y$
		with a function $f=f(y):\; \IR\times \IR^2 \to \IR^2$ by setting $f(t,x) = y(t)(x)$.
		
		Let $e^-_c = (0,-c)$, $e^+ = (0,0)$, $c\geq 0$, $\eps>0$ and
		\begin{align*}
			f^{-\infty}_c(t,(x_1,x_2)) &:= (x_1, -x_2+3x^2_1-c)\\
			f^{\infty}_c(t,(x_1,x_2)) &:= (-x_1 + \eps x_2, x_2)
		\end{align*}

		The global stable and unstable submanifolds are not hard to find. With
		respect to $(\dot x_1, \dot x_2) = f^{-\infty}(0,(x_1, x_2))$ we obtain
		\begin{align*}
			W^u_c(-\infty, e^-_c) &= \{(h, h^2-c):\; h\in\IR\}\\
			W^s_c(-\infty, e^-_c) &= \{0\} \times \IR\\
		\end{align*}
		and with respect to $(\dot x_1, \dot x_2) = f^\infty_c(0,(x_1, x_2))$
		\begin{align*}
			W^u_c(\infty, e^+_c) &= \{ (\eps/2\,h, h) :\; h\in\IR\}\\
			W^s_c(\infty, e^+_c) &= \IR\times\{0\}\\
		\end{align*}
		
		Let $\phi:\IR\to \left[0,1\right]$ be continuously differentiable
		with 
		\begin{equation*}
			\phi(t) = \begin{cases} 0 &t\leq 0\\
			t & t\geq 1
			\end{cases}
		\end{equation*}
		$\phi$ serves as auxiliary function to define $f$ and $g$.
		\begin{align*}
			f_c(t,x) &:= \dot \phi(-t) f^{-\infty}_c(0,x) + \dot \phi(t) f^{\infty}_c(0,x)\\
			g_c(t,x) &:= \dot \phi(-t) f^\infty_c(0,x) + \dot\phi(t) f^{-\infty}_c(0,x)
		\end{align*}
		
		Let $u:\;\IR\to \IR^2$ be a bounded solution of $\dot x = f_c(t,x)$.
		It follows that $u(t) \to e^\pm_c$ as $t\to\pm\infty$. Moreover, one has
		$u(0) \in W^u_c(-\infty, e^-_c) \cap W^s_c(\infty, e^+_c)$. Hence, there are
		exactly two full bounded solutions of $\dot x = f(t,x)$ if $c>0$, one if $c=0$ and
		none otherwise. 
		
		Let $c>0$ and $u$ be a full bounded solution of $\dot x = f(t,x)$. To prove that $(f,u)$
		is weakly hyperbolic, suppose that $w$ is a full bounded solution of $\dot w = d(f,u)(t)w$.
		It follows that $w(0) \in TW^u_c(-\infty, e^-_c)(u(0))\cap TW^s_c(\infty, e^+)(u(0))$. 
		Letting $u(0) = (x_1, x_2)$, we have $x_1 = \sqrt{c}$
		\begin{equation*}
			TW^u_c(-\infty, e^-_c)(x_1, x_2) = \{ (h,2\sqrt{c}h) :\; h\in\IR\}
		\end{equation*}
		and
		\begin{equation*}
			TW^s_c(\infty, e^-_c)(x_1, x_2) = \IR\times\{0\}
		\end{equation*}
		Hence, $w(0)=0$ implying that $w\equiv 0$.
		
		Analogously, a solution $v:\IR\to \IR^2$ of $\dot x = g_c(t,x)$
		must converge to $e^\mp_c$ as $t\to\pm\infty$ and satisfy $v(0) \in W^u_c(\infty, e^+_c) \cap W^s_c(-\infty, e^-_c) = \{0\}$,
		so there is exactly one bounded solution $v$ of this equation.
		Calculating the tangential spaces as before, we see that $v$ is weakly hyperbolic.
		
		Let $y_{0,c} = y_0(f^c, g^c)$ be defined as in Section \ref{sec:parameter},
		and let $K_c\subset \omega(f_0) \times E$ denote
		the largest compact invariant set. Observe that $K_c$ agrees with \eqref{eq:171123-1318}
		for all parameters $c>0$.
		
		The homology Conley index $\Hom_*\Con(f^c_0, K_c)$ is by continuation
		independent of $c\in\IR$. We have already shown that $K_c$ is bounded for all $c\in\IR$,
		so it is sufficient to prove that $y_{0,c}$ depends in an appropriate way 
		continuously on $c$. To see this, one can simply adapt the
		arguments in the proof of Lemma \ref{le:171115-1506}.
		
		Suppose $\Hom_*\Con(y_{0,c}, K_c)$ is non-trivial. It would be particularly
		non-trivial for some $c<0$ implying that there exists a bounded solution
		$u:\;\IR\to E$ of $\dot x = f_c(t,x)$. As discussed above, such a solution does not exist
		for $c<0$, so $\Hom_*\Con(y_{0,c}, K_c) = 0$.
	\end{section}
	
	\begin{section}{A uniformity property}
		\label{sec:uniformity}
		There are several examples of uniformity making the nonautonomous Conley index useful. A nontrivial
		index with respect to a parameter $y_0$ implies the existence of solutions for every equation described by a $y\in\omega(y_0)$ \cite{article_naci}. Similarly, a nontrivial connecting homomorphism implies
		the existence of a connecting orbit for every equation \cite{article_naci_3}.
		
		In this paper, we consider a more specific situation described in the previous section. Roughly speaking, our claim
		is that the Betti numbers of the homology index of an invariant set cannot be larger than the respective numbers obtained independently for each of the autonomous limit	equations.
		
		Given a set $A\subset Y\times X$, we write
		\begin{align*}
			A(Y_0) &:= \{(x,y)\in A:\; y\in Y_0\}\quad\text{ for } Y_0\subset Y\\
			A(y) &:= \{x:\; (y,x)\in A\}\quad\text{ for } y\in Y
		\end{align*}
		
		\begin{theorem}
			\label{th:171115-1441}
			Let $f,g\in Y$ be asymptotically autonomous. Suppose that $f^t\to f^{\pm \infty}\in Y$ as $t\to\pm \infty$ and $g^t\to f^{\mp\infty}$ as $t\to\pm\infty$.
			
			Let $y_0 := y_0(f,g)$ be defined as in Section \ref{sec:parameter}, so in particular $\omega(y_0) = \Hull(f) \cup \Hull(g)$.
			Let $K\subset (\Hull(f)\cup \Hull(g))\times X$ be a compact isolated invariant set, so $K(f^\infty)$ is a compact isolated invariant set with respect to $f^\infty$. 
			
			Then for all $q\in\IZ$, either $\dim \Hom^\IF_q\Con(y_0, K) = \dim \Hom^\IF_q\Con(f^\infty, K(\{f^\infty\})) = \infty$ or $\dim \Hom^\IF_q\Con(y_0, K) \leq \dim \Hom^\IF_q\Con(f^\infty, K(\{f^\infty\}))<\infty$.
%
		\end{theorem}
		
		Note that the Theorem is only proved for homology with coefficients in a field.
		
		The following lemma will be used to prove Theorem \ref{th:171115-1441}. Recall the
		definition of $C_0\subset \IC$ in Section \ref{sec:parameter}. We have also defined a semiflow
		on $C_0$ by differential equations \eqref{eq:171215-1809a} and \eqref{eq:171215-1809b}. This semiflow
		is denoted by $z^t$ for $z\in C_0$, and a designated initial element in $C_0$ is $z_0 := \frac{1}{2}$.
		
		\begin{lemma}
			\label{le:171115-1442}
			In addition to the hypotheses of Theorem \ref{th:171115-1441}, let $a_n\to\infty$ be a sequence
			of nonnegative real numbers such that $z^{a_n}_0/\abs{z^{a_n}_0} = e^{\imag \pi/2}$ for all $n\in\IN$.
			
			Let $\rho>0$ be small enough that the conclusions of Lemma \ref{le:171115-1506} hold,
			and let $K^\rho$ be given by Lemma \ref{le:171115-1506}.
			
			Then there are $\eps>0$ and regular index pairs $(\underline{N}_1, \underline{N}_2)$ and $(\overline{N}_1, \overline{N}_2)$
			for $(y^\rho_0, K^\rho)$ and $(N_1, N_2)$ for $(f^\infty, K(\{f^\infty\}))$
			such that\footnote{$(\tilde N_1, \tilde N_2)$ is an index pair for $K(f^\infty)$ in the sense of Franzosa and Mischaikow
				or an FM-index pair as defined by Rybakowski.}
			\begin{equation*}
				N_1 = \IR^+\times \tilde N_1\quad N_2 = \IR^+\times \tilde N_2\text{ with }\tilde N_1,\tilde N_2\subset X
			\end{equation*}
			and for all $n\in\IN$ sufficiently large
			\begin{equation}
				\label{eq:171115-1557}
				(\underline{N}_1(a_n), \underline{N}^{-\eps}_2(a_n)) 
				\subset (N_1(a_n), N^{-\eps}_2(a_n))
				\subset (\overline{N}_1(a_n), \overline{N}^{-\eps}_2(a_n)) 
			\end{equation}
		\end{lemma}
		
		\begin{proof}[Proof of Theorem \ref{th:171115-1441}]
			Choosing $\rho>0$ small enough, we can assume that the conclusions of Lemma \ref{le:171115-1506} hold,
			so the invariant set $K_\rho$ is defined. Furthermore, there is an isomorphism $\Hom_*\Con(y^\rho_0, K_\rho) \to \Hom_*\Con(y_0, K)$.
			
			We need to show that $\iota$ can be defined by exploiting the inclusions given by \eqref{eq:171115-1557}.
			The homology index is isomorphic to a direct limit. Let $(M_1, M_2)$
			be a regular index pair for $(y^\rho_0, K)$ and $m,n\in \IN$ with $m\leq n$.
			By Lemma 3.8 in \cite{article_naci_3} and for $\eps>0$,
			there are homomorphisms 
			\begin{equation*}
			g_{m,n}:\;\Hom_*(M_1(\{a_m\}), M^{-\eps}_2(\{a_m\}))\to \Hom^{-\eps}_*(M_1(\{a_n\}), M^{-\eps}_2(\{a_n\})) =: A_n
			\end{equation*}
			The family $(A_n, g_{n,m})$ with the natural ordering on $\IN$ forms a direct system.
			The inclusions $(M_1(\{a_n\}), M^{-\eps}_2(\{a_n\}) \subset (M_1, M^{-\eps}_2)$ induce
			an isomorphism 
			\begin{equation}
				\label{eq:171219-0945}
				\dirlim (A_n, g_{n,m}) \to \Hom_*(M_1, M^{-\eps}_2).
			\end{equation}
			
			The direct limit $\dirlim(A_n, g_{n,m})$ is a set of equivalence classes, where $\rep{\alpha} = \rep{\beta}$
			if $\alpha\in A_n, \beta\in A_m$, $n\leq m$ and $g_{n,m}(\alpha) = \beta$. Using \eqref{eq:171219-0945},
			we can identify $\rep{\alpha}$ with an element of $\Hom_*(M_1, M^{-\eps}_2)$.
			
			Let
			 $(\eta_i)_{i\in I}$ be a basis for $\Hom^\IF_q(\underline N_1,\linebreak[0] \underline N^{-\eps}_2)$.
			Each $\eta_i$ can be written as $\eta_i = \rep{\alpha_i}$ with $\alpha_i\in \Hom_q(\underline{N}_1(\{a_n\}), \underline{N}^{-\eps}_2(\{a_n\})$. Suppose that $I_0\subset I$ is a finite subset of $k_0$ elements.
			By using our freedom to choose $\alpha_i$, we can assume without loss of generality that
			$\alpha_i\in \Hom_q(\underline{N}_1(\{a_{n_0}\}))$ for all $i\in I_0$ and some $n_0\in\IN$ independent
			of $i$.	

			Let $\mathcal{A}$ denote the subspace of $\Hom_q(\underline{N}_1(\{a_n\}), \underline{N}^{-\eps}_2(\{a_n\}))$
			spanned by $(\alpha_i)_{i\in I_0}$ and $\alpha\in \mathcal{A}$ be arbitrary. Let 
			$k_q:\; \Hom_q(\underline{N}_1(\{a_n\}), \underline{N}^{-\eps}_2(\{a_n\}))\supset \mathcal{A} \to \Hom_q(N_1(\{a_n\}), N^{-\eps}_2(\{a_n\}))$ and
			$l_q:\; \Hom_q(N_1(\{a_n\}), N^{-\eps}_2(\{a_n\})) \to \Hom_q(\overline{N}_1(\{a_n\}), \overline{N}^{-\eps}_2(\{a_n\}))$ be inclusion induced. It follows from \eqref{eq:171115-1557}
			that $\rep{l_q\circ k_q(\alpha)} = j_q(\rep{\alpha})$, where $j_q:\; \Hom_q(\underline{N}_1, \underline{N}^{-\eps}_2)
			\to \Hom_q(\overline{N}_1, \overline{N}^{-\eps}_2)$ is inclusion induced and an isomorphism \cite[Lemma 2.5]{article_naci_1} since both
			$(\underline{N}_1, \underline{N}^{-\eps}_2)$ and $(\overline{N}_1, \overline{N}^{-\eps}_2)$ are index pairs
			for $(y^\rho_0, K_\rho)$.
			Therefore, $k_q(\alpha)\neq 0$ unless $\alpha=0$, 
			implying that $\dim \Hom^\IF_q(N_1(\{a_{n_0}\}), N^{-\eps}_2(\{a_{n_0}\}))\geq k_0$.
			
			Note that $N_i(\{a_n\}) = \{a_n\}\times \tilde N_i$ for $i\in\{1,2\}$, where $\tilde N_1$
			and $\tilde N_2$ are given by Lemma \ref{le:171115-1442}. Thus, $\Hom_*(N_1(\{a_n\}, N^{-\eps}_2(\{a_n\}))$
			and $\Hom_*(N_1, N^{-\eps}_2)$ are isomorphic regardless of $n$. The latter is an index pair for $(f^\infty, K(\{f^\infty\}))$.
		\end{proof}
		
		To prove Lemma \ref{le:171115-1442}, another auxiliary lemma is required. The index pairs
		will be obtained using a method developed in \cite{article_naci_1}. For the
		reader's convenience, we will recall the construction shortly.
		
		Let $y_0\in Y$ and $N\subset \Hull^+(y_0)\times X$ be a strongly admissible isolating neighbourhood
		for an invariant set $K\subset \omega(y_0)\times X$. Let $U$ be an open subset of $\Hull^+(y_0)\times X$ 
		with $K\subset U\subset N$, and set
		\begin{align*}
			g^+(y,x) &:= \sup\{t\in\IR^+:\; (y,x)\pi\left[0,t\right]\subset U\}\\
			g^-(y,x) &:= \sup\{d((y,x)\pi t, \Inv^-_\pi(N)):\; t\in\left[0,g^+(y,x)\right]\}
		\end{align*}
		
		The sets $\{g^+\leq c\} := \{(y,x)\in N:\; g^+(y,x)\leq x\}$ and $\{g^-\leq c\} := \{(y,x)\in N:\; g^-(y,x)\leq c\}$
		are closed for arbitrary $c\in\IR^+$ by Lemma 2.1 in \cite{article_naci_1}. Additionally, it is proved
		\cite[Lemma 2.12]{article_naci_1} that $(\hat L^{c_1,c_2}_1, \hat L^{c_1,c_2}_2)$ is an index pair for $(y_0, K)$,
		where we set
		\begin{align*}
			L^{c_1,c_2}_1 &:= \{g^- \leq c_1\} \cap \cl\{g^+\geq c_2\}\\
			L^{c_1,c_2}_2 &:= L^{c_1, c_2}_1 \cap \{g^+\leq c_2\}
		\end{align*}
		$r_{y_0}:\;\IR^+\times X \to Y\times X$, $r_{y_0}(t,x) = (y^t_0,x)$
		and $\hat L^{c_1,c_2}_i := r^{-1}_{y_0}(L^{c_1,c_2}_i)$.
		
		\begin{lemma}
			\label{le:171116-1805}
			Let us make the following assumptions.
			\begin{enumerate}
				\item $y_0,\tilde y_0\in Y$
				\item $a_n\to\infty$ is a sequence of nonnegative reals
				\item $(N_1, N_2)$ is an index pair for $(\tilde y_0,\tilde K)$
				\item there exists a strongly admissible isolating neighbourhood
					for $K$ in $\Hull^+(y_0)\times X$
				\item $y^{a_n}_0\to f^\infty$ and $\tilde y^{a_n}_0\to f^\infty$ as $n\to\infty$,
				where $f^\infty\in Y$ is autonomous as previously assumed
				\item $K(f^\infty) = \tilde K(f^\infty)$
				\item $T>0$ is a real number
			\end{enumerate}
			
			Then there exists a strongly admissible isolating neighbourhood $N$ and
			an open set $U$ with $K\subset U\subset N$
			such that for all $c_1>0$ sufficiently small, all $T'>0$ sufficiently large and all $n\in\IN$
			large enough, it holds that
			\begin{equation*}
				(\hat L^{c_1,c_2}_1(a_n), (\hat L^{c_1,c_2}_2)^{-T}(a_n)) \subset (N_1(a_n), N^{-T'}_2(a_n))
			\end{equation*}
		\end{lemma}
		
		\begin{proof}
			$(N_1, N_2)$ is an index pair for $(\tilde y_0, \tilde K)$, so by definition exists
			an open neighbourhood $\tilde U\subset \Hull^+(\tilde y_0)$ of $\tilde K$ such that
			$r^{-1}_{\tilde y_0}(\tilde U)\subset N_1\setminus N_2$. 
			
			By choosing $N$ small enough, we can assume that $N(y)\subset \tilde U(y)$ for
			all $y$ in a small neighbourhood of $f^\infty$. It follows that $\hat L^{c_1,c_2}_1(a_n)\subset N_1(a_n)$
			for $n\in\IN$ large enough and regardless of $c_1, c_2$. Fix an open set $U$ with $K\subset U\subset N$.
			
			Assume for contradiction there are sequences $c^n_1\to 0$
			and $x_n\in (\hat L^{c^n_1,c_2})^{T'}_2(a_n)$ but $x\chi_{\tilde y_0} t\in N_1(a_n)\setminus N_2(a_n)$ for all $t\in\left[0,n\right]$. It follows that $d((y^{a_n}_0,x_n), \Inv^-_\pi(N))\to 0$, so we may assume
			without loss of generality that $(y^{a_n}_0,x_n)\to (f^\infty,x_0)\in \Inv^-_\pi(N)$ since $\Inv^-_\pi(N)$ is compact. We have $N(f^\infty)\subset \tilde N(f^\infty)$, so $(f^\infty,x_0)\in \Inv^-_\pi(\tilde N)$.
			
			The definition of an index pair implies that there is a strongly admissible isolating neighbourhood
			$\tilde N$ for $(\tilde y_0, \tilde K)$ such that $N_1\setminus N_2\subset r^{-1}_{y_0}(\tilde N)$.
			We have $(\tilde y^{a_n}_0, x_n)\pi t\in \tilde N$ for all $t\in\left[0,n\right]$, so
			$(f^\infty,x_0)\in \Inv^+_{\pi}(\tilde N)$.
			
			Together, we have proved that $(f^\infty,x_0)\in \Inv(\tilde N)$, so $(f^\infty, x_0)\in \tilde K(f^\infty) = K(f^\infty)$.
			
			Additionally, it holds that $(y^{a_n}_0, x_n)\pi s\in L^{c_1,c_2}_2$ for some $s_n\in\left[0,T\right]$, so $(f^\infty, x_0)\pi s\in L^{c_1, c_2}_2$ for some $s\in \left[0,T\right]$. However, it follows from Lemma 2.11 (e) in \cite{article_naci_1} that
			$K\cap L^{c_1,c_2}_2 = \emptyset$, a contradiction.
		\end{proof}
		
		\begin{proof}[Proof of Lemma \ref{le:171115-1442}]
			Let $(\overline{N}_1, \overline{N}_2)$ be an index pair for $(y^\rho_0, K)$. An index pair
			can be obtained by using the construction above, see \cite[Lemma 2.12]{article_naci_1}. Renaming
			$(y^\rho_0, K)$ to $(\tilde y_0, \tilde K)$ and replacing $(y_0, K)$ by $(f^\infty, K(\{f^\infty\}))$,
			we can apply Lemma \ref{le:171116-1805} and obtain an index pair $(N_1, N_2)$ for $(f^\infty, K(\{f^\infty\}))$
			such that $(N_1(a_n), N_2(a_n)) \subset (\overline{N}_1(a_n), \overline{N}^{-T}_2(a_n)$ for some $T>0$
			and almost all $n\in\IN$.
			
			By using Lemma \ref{le:171116-1805} again, we show that there exists an index pair $(\underline{N}_1, \underline{N}_2)$ for $(y^\rho_0, K)$ such that $(\underline{N}_1(a_n), \underline{N}^{-\eps}_2(a_n)) \subset (N_1(a_n), N^{-T'}_2(a_n))$ for some $T'>0$	and almost all $n\in\IN$.
			
			Recall that by Lemma 2.7 in \cite{article_naci_1}, 
			$(M_1, M^{-T}_2)$ is an index pair for $(y'_0, K')$ if $(M_1, M_2)$ is an index pair for $(y'_0, K')$.
			Let $\Delta>0$ be arbitrary.
			The evolution operators $\Phi_{y^\rho_0}(t,a_n,.)$
			and $\Phi_{f^\infty}(t,a_n,.)$ agree for $t-a_n\leq \Delta$ provided that $n$ is large enough
			since $F^\rho(r \mathrm{e}^{\imag \varphi}) = f^\infty(0)$ in a sector of angle $\rho$ around $\varphi_0 = \pi/2$.
			Hence, $\left(N_1(a_n), N^{-\Delta}_2(a_n)\right)\subset \left(\overline{N}_1(a_n), \overline{N}^{-(T+\Delta)}_2(a_n)\right)$
			provided that $n\geq n_0(\Delta)$. Choosing $\Delta:=T'$, one has
			\begin{align*}
				\left(\underline{N}_1(a_n), \underline{N}^{-\eps}_2(a_n)\right) 
				&\subset \left(N_1(a_n), \left(N^{-(\Delta-\eps)}_2\right)^{-\eps}(a_n)\right)\\
				&\subset \left(\overline{N}_1(a_n), \left(\overline{N}^{-(T+\Delta-\eps)}_2\right)^{-\eps}(a_n)\right)
			\end{align*}
			for all but finitely many $n\in\IN$.
		\end{proof}
	\end{section}
	
	\begin{section}{Abstract semilinear parabolic equations}
		\label{sec:abstracteq}
		In this section, we will prove results concerning the existence and
		multiplicity of hyperbolic solutions, in particular
		an abstract version of Theorem \ref{th:171208-1449} which
		can be found in the introduction.
		
		The abstract setting follows Section \ref{sec:parameters_infinite}, so in particular
		$E:=W^\alpha$, where $W^\alpha$ denotes the $\alpha$-th fractional power space. 
		Let us assume that the parameter space $Y$ ist the set of all continuous $f:\IR\times E\to W$
		which agrees up to an obvious identification with the definition given in Section \ref{sec:parameters_infinite}.
		
		We consider evolution operators $\Phi_f$ defined by mild solution of
		the following evolution equation.
		\begin{equation}
			\label{eq:171208-1539}
			\dot u + Au = f(t,u)\quad f\in Y
		\end{equation}
		
		Assume that $f$ is asymptotically autonomous with $f^t\to f^{\pm\infty}$ as $t\to\pm\infty$. Let
		$g$ be another asymptotically autonomous parameter with $g^t\to f^{\mp\infty}$ as $t\to\pm\infty$.
		The time reverse of $f$ i.e., $g(t,x) = f(-t,x)$ for all $(t,x)\in\IR\times E$, is a possible
		choice for $g$, but this approach is sometimes too simple.
		Using $y_0:=y_0(f,g)$ as defined in Section \ref{sec:parameter}, we have
		$\Hull^+(y_0) = \Hull(f)\cup\Hull(g)$.
		
		In addition to the regularity assumptions on $f$ imposed by $Y$, let us assume
		that $f^{\pm\infty}(0,x)\in C^1(E,W)$. Consequently, the Morse-index $m(e)$ can be 
		defined for every hyperbolic equilibrium of
		\begin{equation}
			\label{eq:171208-1647}
			\dot u + Au = f^{\pm\infty}(0,u)
		\end{equation}
		
		Let $K_0\subset E$ be an isolated invariant subset with respect to the
		semiflow induced by \eqref{eq:171208-1647} for either of the equations. 
		We say that $f$ is asymptotically gradient-like if:
		\begin{enumerate}
			\item[(H1)] Every equilibrium of \eqref{eq:171208-1647} is hyperbolic.
			\item[(H2)] A solution $u:\;\IR\to K_0$ of $\Phi_{f^{\pm\infty}}$
			is either constant or there are equilibria $e^\pm$ of $\Phi_{f^{\pm\infty}}$
			such that $m(e^-) > m(e^+)$ and	$u(t)\to e^\pm$ as $t\to\pm\infty$.
		\end{enumerate}
		
		In view of \cite[Theorem 2.4]{article_generic17}, we say that $f$ is gradient-like
		if $f$ is asymptotically gradient-like and:
		\begin{enumerate}
			\item[(H3)] For every solution $u:\IR\to K(\Hull(f))$, there are $(f^-,e^-), (f^+,e^+)\in K$
			such that $m(e^-)\geq m(e^+)$ and $u(t) \to (f^\pm, e^\pm)$ as $t\to\infty$.
		\end{enumerate}
		
		\begin{theorem}
			\label{th:171208-1837}
			In addition to the above hypotheses, assume that $f$ and $g$ are gradient-like.

			If $K\subset\Hull^+(y_0)\times E$ is a compact isolated invariant set with
			\begin{equation*}
				\dim\Hom^\IF_q\Con(y_0(f,g), K)=n_0 \quad \text{for some }q\in\IZ
			\end{equation*}
			then there are (at least) $n_0$ pairwise distinct solutions $u_1,\dots, u_{n_0}$ of $\Phi_f$ connecting
			equilibria of Morse-index $q$.
		\end{theorem}
		
		\begin{proof}
			Let $M_k$ denote the set of all $(y,x)\in K$ such that there are a solution $u:\IR\to \omega(y_0)\times X$
			and equilibria $e^-, e^+\in E$ such that $(f^-,e^-), (f^+,e^+)\in \{f^{-\infty},f^\infty\}\times E$, $m(e^-) = m(e^+)=k$ and $u(t)\to (f^\pm,e^\pm)$ as $t\to\pm\infty$.

			Every equilibrium is hyperbolic by assumption, whence it follows that $M_k = \emptyset$ for all $k\in\IN$
			sufficiently large. Moreover,
			$f$ and $g$ are both gradient-like, so the family $(M_k)_{k\in\{1,\dots,k_0\}}$ is a Morse-decomposition
			of $K$. 
			
			We will prove the theorem by induction on the number of Morse sets $M_k$. 
			Assume there are less than $\overline{n}_0$ pairwise distinct solutions connecting equilibria of Morse-index $q$. We need to show that $\dim\Hom^\IF_q\Con(f_0, K) < \overline{n}_0$.
			
			First, the case where $K$ is equal to $M_{k_0}$ for some $k_0\in\IN$. Let $\mathcal{E}$ denote the set of
			equilibria in $M_{k_0}$ at $f^\infty$ i.e., $\mathcal{E} = M_{k_0}(f^\infty)$. $\mathcal{E}$ can be split into those having
			a past connection and those not having a past connection. We say that $e\in \mathcal{E}$ has a past connection
			if there is a solution $u:\;\IR\to M_{k_0}$ such that $u(t)\to (f^\infty, e)$ as $t\to\infty$.
			
			We write $\mathcal{E}_0$ to denote the set of all equilibria not having a past connection. It follows that
			$\{f^\infty\}\times \mathcal{E}_0$ is a repeller in $M_{k_0}$. The uniformity property of the index \cite[Corollary 4.10]{article_index} implies
			that $\Hom^\IF_q\Con(y_0, \{f^\infty\}\times E) = 0$. Let $A$ denote the attractor associated with $\mathcal{E}_0$,
			so $(f_0, M_{k_0}, A, \mathcal{E}_0)$ is an attractor-repeller decomposition, giving rise to a long exact sequence.
			\begin{equation*}
				\xymatrix@1{
					\ar[r] &\relax\underbrace{\Hom^\IF_q\Con(y_0, \mathcal{E}_0)}_{=0} \ar[r] &\Hom^\IF_q\Con(y_0, M_{k_0}) \ar[r]
					&\ar[r] \Hom^\IF_q\Con(y_0, A) \ar[r]&
					}
			\end{equation*}
			Consequently, $\Hom^\IF_q\Con(y_0, M_{k_0})$ and $\Hom_q\Con(y_0, A)$ are isomorphic. In view
			of Theorem \ref{th:171115-1441}, $\dim\Hom^\IF_q\Con(y_0, A) \leq \dim \Hom^\IF_q\Con(f^\infty, \mathcal{E}^+)$, where $\mathcal{E}^+$
			is the set of all equilibria having a past connection. It well-known that $\dim\Hom_q\Con(f^\infty, \mathcal{E}^+)$
			is equal to the number of equilibria in $\mathcal{E}^+$ which cannot be larger than the upper bound $n_0$ on the number of
			distinct connections.
			
			Let $A_q$ consist of all $M_k$ for $k\leq q$ and all connecting orbits from $M_k$ to $M_l$
			where $l<k\leq q$. The corresponding repeller is $R_q$ consisting of all $M_k$ where $k>q$
			and all connecting orbits from $M_k\to M_l$ where $k>l>q$.
			
			If $q<k_0$, it is known by induction that $\Hom^\IF_q\Con(y_0, R_q) = 0$ and $\dim \Hom^\IF_q\Con(f_0, A_q)<\overline{n}_0$.
			The attractor-repeller decomposition gives rise to a long exact sequence \cite{article_naci_1}
			\begin{equation*}
				\xymatrix@1{
					\ar[r] &\Hom^\IF_q\Con(y_0,A_q) \ar[r]^-{i_q}
					& \Hom^\IF_q\Con(y_0, K) \ar[r]
					& \relax\underbrace{\Hom^\IF_q\Con(y_0, R_q)}_{=0} \ar[r] &
				}
			\end{equation*}
			$i_q$ must be surjective and hence $\dim\Hom^\IF_q\Con(f_0,K) \leq \Hom^\IF_q\Con(f_0, A_q) < \overline{n}_0$.
			
			If $q=k_0$, we take $A_{q-1}$ and $R_{q-1} = M_q$ instead of $A_q$ and $R_q$. 
			The long exact attractor-repeller sequence	reads as follows.
			\begin{equation*}
			\xymatrix@1{
				\ar[r] & \relax\underbrace{\Hom^\IF_q\Con(y_0,A_{q-1})}_{=0} \ar[r]
				& \Hom^\IF_q\Con(y_0, K) \ar[r]^-{p_q}
				& \Hom^\IF_q\Con(y_0, M_q) \ar[r] &
			}
			\end{equation*}
			$p_q$ must be injective, so $\dim\Hom^\IF_q\Con(y_0, K)\leq \dim\Hom^\IF_q\Con(y_0, M_q) < \overline{n}_0$.
		\end{proof}

		In the corollary below, we have merely replaced the assumption that $g$ is gradient-like
		by a regularity assumption.
		The rather obvious strategy for its proof is to perturb $g$ to enable the use of Theorem \ref{th:171208-1837}.

		\begin{corollary}
			\label{co:171212-1331}
			In addition to the above hypotheses, assume that $E$ is a reflexive space,
			$f$ is gradient-like $g\in C^\infty(\IR\times E, W)$.
			
			If $K\subset\Hull^+(y_0)\times E$ is a compact isolated invariant set with
			\begin{equation*}
			\dim\Hom_q\Con(y_0(f,g), K)=n_0 \quad \text{for some }q\in\IZ
			\end{equation*}
			then there are $n_0$ pairwise distinct solutions $u_1,\dots, u_{n_0}$ of $\Phi_f$
			connecting equilibria of Morse-index $q$.
		\end{corollary}
		
		\begin{proof}
			In view of Theorem 2.4 in \cite{article_generic17}, we can find an arbitrarily small perturbation $\tilde g$
			of $g$ such that $g$ is gradient-like and $(\tilde g - g)^t\to 0$ as $t\to\pm\infty$. 
			Let $y_0:=y_0(f,g)$ denote the initial element 
			with respect to $(f,g)$ and $\tilde y_0 := y_0(f, \tilde g)$ the initial element obtained with respect
			to $(f,\tilde g)$. Using almost the same arguments as in the proof of Lemma \ref{le:171115-1506},
			we see that $\Hom_*\Con(y_0, K) \iso \Hom_*\Con(\tilde y_0, \Inv N)$ provided that
			$\tilde g$ is sufficiently close to $t$ in $Y$, where $N\subset Y\times X$
			is an isolating neighbourhood for $K$.
			
			The result now follows directly from Theorem \ref{th:171208-1837}.
		\end{proof}
		
		For the statement of the following theorem, we will assume that 
		\begin{enumerate}
			\item[(REF)] $E$ is a reflexive space
			\item[(LIN)]
		$f$ is asymptotically linear. Additionally, we will also assume that
		the asymptotical linearity is autonomous. More precisely, let $B\in\mathcal{L}(E,W)$
		be such that 
		\begin{equation*}
			\eps f(t,\eps^{-1} u) - Bu \to 0 \text{ as }\eps \to 0
		\end{equation*}
		uniformly on $\IR\times B_1(0;E)$. Let $m(B)$ denote the dimension
		of the generalized eigenspace associated with the negative eigenvalues
		of $A-B$. $m(B)$ is always finite since $A$ is assumed to have compact resolvent.
		\end{enumerate}
		
		Let $\mathcal{F}$ denote the set of all $f\in C^\infty(\IR\times E, W)$ such that:
		\begin{enumerate}
			\item $f$ is asymptotically autonomous
			\item $f(t,.)\to f^{\pm\infty}(0,.)$ in $C^1(E,W)$ as $t\to\pm\infty$
			\item (LIN) holds
		\end{enumerate}
		
		We treat $\mathcal{F}$ as a subspace of $C^1(\IR\times E, W)$, where $C^1(\IR\times E, W)$
		is equipped with a metric such that a $f_n\to f$ in $C^1(\IR\times E, W)$ iff $f_n\to f$ and $\Diff f_n\to \Diff f$
		uniformly on sets of the form $\IR\times B_\eps(0;E)$ for $\eps>0$.
		
		\begin{theorem}
			\label{th:171213-1457}
			For a generic $f\in\mathcal{F}$, 			
			there exists a (weakly) hyperbolic solution $u:\IR\to E$ of 
			\begin{equation*}
				\dot u + Au = f(t,u)
			\end{equation*}
			and equilibria $e^\pm$ of Morse-index $m(B)$ of
			\begin{equation*}
				\dot u + Au = f^{\pm\infty}(0,u)
			\end{equation*}
			such that $u(t)\to e^\pm$ in $E$ as $t\to\pm\infty$.
		\end{theorem}
		
		It is sufficient to prove the theorem with respect to the existence
		of weakly hyperbolic solutions as it follows from \cite[Theorem E]{sacker_sell_dich}
		that a weakly hyperbolic solution satisfying the conclusion
		of the theorem is hyperbolic.
		
		\begin{proof}
			Let $f\in \mathcal{F}$ be arbitrary and take $g$ to be
			the time inverse introduced at the beginning of this section. 
			It is clear that $g$ is also asymptotically linear with respect
			to the same asymptotic linearity as $f$.
			
			It follows from \cite[Theorem 7.16]{article_index} that
			there exists a largest compact invariant subset $K$ of $\Hull^+(y_0(f,g))$
			having the following homology Conley index.
			\begin{equation}
				\label{eq:171212-1338}
				\Hom^\IF_q\Con(y_0(f,g), K) \iso \begin{cases} \IF & q=m(B) \\ 0 & q\neq m(B)
				\end{cases}
			\end{equation}
			
			It follows as in the proof of Corollary \ref{co:171212-1331}
			that there exists an arbitrarily small perturbation $\tilde f$
			of $f$ such that (H3) is fulfilled with respect to $\tilde f$. Let $\tilde y_0:=y_0(\tilde f,g)$
			be an initial element for $(\tilde f, g)$.
			
			Moreover for sufficiently small perturbations, 
			$\Con(y_0(f,g),K)$ continues to $\Con(\tilde y_0, \tilde K)$.
			It now follows directly from Corollary \ref{co:171212-1331} that there exists a solution $u$ 
			as claimed.

			Finally, we need to prove that the set of all $f\in \mathcal{F}$ for
			which there exists such a solution is open. We can assume without loss
			of generality that the hyperbolic solution is $u\equiv 0$. Define an
			operator 
			\begin{equation*}
				G:\; \underbrace{C^{1,\delta}(\IR,W) \cap C^{0,\delta}(\IR, W^1)}_{=:\mathcal{X}}\to 
				C^{0,\delta}(\IR,W)
			\end{equation*}
			by 
			\begin{equation*}
				G(u)(t) := u_t(t) - Au(t) + f(t,u(t))
			\end{equation*}
			
			$G$ is continuously Fr\'echet-differentiable (see e.g. \cite[Lemma 4.2]{article_generic17})
			and $L:=\Diff G(0)$ is invertible \cite[Theorem 4.a.4]{brunpol}.
			
			Let $\tilde f\in\mathcal{F}$ be a perturbation of $f$. 
			We seek for a solution $\tilde u$ of
			\begin{equation*}
				\tilde u_t(t) - A\tilde u(t) + \tilde f(t,u(t)) = 0
			\end{equation*}
			or equivalently
			\begin{equation*}
				Lu + \tilde f(t,u) - \Diff_x f(t,0)u(t) = 0.
			\end{equation*}
			
			Choosing $\eps>0$ small enough,
			\begin{equation}
				\label{eq:171213-1354}
				L^{-1}\left(\tilde f(t,u) - \Diff_x f(t,0)u(t)\right)
			\end{equation}
			is a contraction mapping on $B_\eps(0;\mathcal{X})$ provided that $d(\tilde f,f)<\eps$, 
			whence it follows
			that there exists a unique fixed point in $B_\eps(0;\mathcal{X})$ provided
			$\tilde f - f$ is sufficiently small in $\mathcal{F}$.
			
			Let $\tilde f_n$ be a sequence in $B_\eps(f;\tilde F)$ with
			$\tilde f_n\to f$, and let $u_n\in B_\eps(0;\mathcal{X})$ denote
			the unique solution of \eqref{eq:171213-1354}. For each $n\in\IN$,
			there are equilibria $e^\pm_n$ such that $u_n(t)\to e^\pm_n$ as $t\to\pm\infty$.
			We have $e^\pm_n\to 0$ as $n\to\infty$, so for large $n\in\IN$, $e^\pm_n$
			is a hyperbolic equilibrium having Morse-index $m(B)$. Furthermore,
			$u_n\to 0$ uniformly on $\IR$ implies that $u_n$ 
			must be a weakly hyperbolic solution for almost all $n$.
			
			It follows by contradiction that for all $\tilde f\in \mathcal{F}$ situated in a sufficiently
			small neighbourhood of $f$, there exists a weakly hyperbolic solution $u$ with 
			the stated properties.
			
		\end{proof}

	\end{section}


\end{document}